\documentclass[11pt]{amsart}
\usepackage{amsmath,amsfonts,amssymb,enumerate,mathrsfs,mathtools,amscd}

\newcommand{\R}{\mathbb{R}}                   
\newcommand{\C}{\mathbb{C}}                   
\newcommand{\G}{\Gamma}
\newcommand{\Sg}{\Sigma}

\newcommand{\Reg}{\mathrm{Reg}}
\newcommand{\Sing}{\mathrm{Sing}}
\newcommand{\Int}{\mathrm{Int}}
\newcommand{\vp}{\varphi}

\newcommand{\An}{\mathscr{A}}                 
\newcommand{\e}{\varepsilon}

\theoremstyle{plain}
\newtheorem{theorem}{Theorem}[section]
\newtheorem{proposition}[theorem]{Proposition}
\newtheorem{lemma}[theorem]{Lemma}
\newtheorem{corollary}[theorem]{Corollary}

\theoremstyle{definition}

\newtheorem{example}[theorem]{Example}
\newtheorem{remark}[theorem]{Remark}

\numberwithin{equation}{section}

\begin{document}

\title[Arc-analytic subanalytic functions on complex manifolds]{Arc-analytic subanalytic functions on complex manifolds}

\author{Janusz Adamus}
\address{Department of Mathematics, The University of Western Ontario, London, Ontario, N6A 5B7 Canada}
\email{jadamus@uwo.ca}
\author{Rasul Shafikov}
\address{Department of Mathematics, The University of Western Ontario, London, Ontario, N6A 5B7 Canada}
\email{shafikov@uwo.ca}
\thanks{Research was partially supported by the Natural Sciences and Engineering Research Council of Canada}


\begin{abstract}
We show that an arc-analytic subanalytic function on a complex manifold $M$, which is holomorphic near one point, is a holomorphic function on $M$.
More generally, an arc-analytic subanalytic function on a real analytic CR-manifold $M$, which is CR on a nonempty open subset of $M$, is a CR-function on the whole $M$.
\end{abstract}
\maketitle


\section{Introduction}
\label{sec:intro}

The purpose of this article is to study the behaviour of arc-analytic functions with subanalytic graphs defined on complex (or, more generally, CR-) manifolds.
The arc-analytic subanalytic functions play an important role in real analytic geometry, thanks to the fundamental work of Bierstone-Milman \cite{BM2}. It is well known that an arc-analytic subanalytic function $f$ on a real analytic manifold $M$ is real analytic outside a closed subanalytic subset $S(f)\subset M$, and it follows from \cite{BM2} (or \cite{Par}) that $S(f)$ has codimension at least 2 in $M$. However, the global behaviour of these functions is rather poorly understood.

By contrast, the present note shows that, when defined over complex- or CR-manifolds, the local regularity of such functions has profound global consequences.
The following are our main results:

\begin{theorem}
\label{thm:glob-holo-map}
Let $M$ be a connected complex manifold, let $N$ be a complex manifold, and let $f:M\to N$ be an arc-analytic subanalytic mapping. Suppose that $f$ has a holomorphic germ $f_p$, for some $p\in M$. Then, $f$ is a holomorphic map on $M$.
\end{theorem}

\begin{theorem}
\label{thm:glob-CR-map}
Let $M$ be a connected real analytic CR manifold, and let $f:M\to\C^N$ be an arc-analytic subanalytic map. Suppose that $f$ is CR on some nonempty open $U\subset M$. Then $f$ is a CR map on $M$.
\end{theorem}
\smallskip

A considerably weaker, semialgebraic variant of the above theorem was proved in \cite[Thm.\,1.1]{A1}. Without arc-analyticity, but under additional assumptions that $f$ be continuous and the graph of $f$ be real analytic, it is \cite[Thm.\,1.3]{Sh}.
\smallskip

Recall that a subset $X$ of a real analytic manifold $M$ is called \emph{semianalytic}, when every point $p\in M$ has an open neighbourhood $U$ such that $X\cap U$ is a finite union of sets of the form
\[
\{x\in U\;:\;f(x)=0, g_1(x)>0,\dots,g_k(x)>0\}\,,
\]
where $f,g_1,\dots,g_k\in\An(U)$ are real analytic functions on $U$. A set $Y\subset M$ is called \emph{subanalytic}, when for every point $p\in M$ there are an open neighbourhood $U$ and a semianalytic set $X$ in another manifold $N$, such that $Y\cap U=\pi(X)$, where $\pi:N\to M$ is a proper real analytic map.

Let $X\subset M$ be non-empty. A function $f:X\to\R^n$ is called \emph{arc-analytic}, if $f\circ\gamma$ is a real analytic function for every analytic arc $\gamma:(-1,1)\to X$. It is called a \emph{subanalytic function}, if its graph $\G_f$ is a subanalytic subset of $M\times\R^n$.
\medskip

For definitions of CR manifolds, CR functions, and basic facts from CR geometry, see Section~\ref{sec:background}. Section~\ref{sec:background} also contains a review of relevant results from arc-analytic  and C-analytic geometry. The proofs of Theorem~\ref{thm:glob-holo-map} and related results are given in Section~\ref{sec:pf-thm-holo}. Section~\ref{sec:pf-thm-CR} contains the proof of Theorem~\ref{thm:glob-CR-map}. In the final section, we discuss several examples of arc-analytic functions with subanalytic graphs, that are interesting in the context of the above theorems.


\section{Background}
\label{sec:background}

\subsection*{C-analytic sets}

Let $\Omega$ be a real analytic manifold, and let $\Omega^*$ denote its complexification (see, e.g., \cite{GMT} for a modern exposition of complexification of real analytic spaces). A set $R\subset\Omega$ is called \emph{C-analytic}, when there exists an open neighbourhood $V^*$ of $\Omega$ in $\Omega^*$ and a complex analytic set $Z$ in $V^*$ such that $Z\cap\Omega=R$ (see, e.g., \cite{WB}). 
The following properties of C-analytic sets will be used throughout the paper.

\begin{remark}
\label{rem:C-analytic}
{~}
\begin{itemize}
\item[(i)] By \cite[Prop.\,10]{WB} (cf. \cite[Prop.\,15]{Car}), $R$ is a C-analytic subset of $\Omega$ if and only if $R$ can be realized as the common zero locus of finitely many real analytic functions on $\Omega$, and thus $R=f^{-1}(0)$ for some $f\in\An(\Omega)$. In particular, every point $p$ of a real analytic set $R\subset\Omega$ admits an open neighbourhood $U^p$ such that $R\cap U^p$ is C-analytic in $U^p$.
\item[(ii)] For a set $S\subset\Omega$, its \emph{C-analytic closure} is the smallest C-analytic set in $\Omega$ which contains~$S$. It is well defined, as the intersection of any family of C-analytic sets is itself C-analytic (see, e.g., \cite[\S\,8]{WB}).
\item[(iii)] A C-analytic set $R\subset\Omega$ is called \emph{C-irreducible} if it cannot be represented as a union of two proper C-analytic subsets in $\Omega$. By \cite[Prop.\,11]{WB}, every C-analytic set in $\Omega$ admits a locally finite decomposition into C-irreducible C-analytic sets, called its \emph{C-irreducible components}.
\item[(iv)] By the corollary following \cite[Prop.\,12]{WB}, if $R$ is a C-irreducible subset of $\Omega$ and $S\subset R$ is C-analytic in $\Omega$, then $\dim_\R{S}<\dim_\R{R}$ or else $S=R$.
\end{itemize}
\end{remark}
\smallskip

\subsection*{Arc-analytic maps and arc-symmetric sets}

Let $M$ be a real analytic manifold. A set $E\subset M$ is called \emph{arc-symmetric}, when for every analytic arc $\gamma:(-1,1)\to M$ with $\Int(\gamma^{-1}(E))\neq\varnothing$, one has $\gamma((-1,1))\subset E$. (Here, $\Int(S)$ denotes the interior of a set $S$.) Arc-symmetric sets were first introduced and studied by Kurdyka in his seminal paper \cite{Ku}, in the semialgebraic setting. Recently, their main algebro-geometric properties have been generalized to the semianalytic and subanalytic setting by the first signed author in \cite{A2}. In the present note, we shall only need the following basic properties of these sets.

\begin{remark}
\label{rem:AR-facts}
{~}
\begin{itemize}
\item[(i)] Subanalytic arc-symmetric sets are closed in Euclidean topology. This follows from the subanalytic Curve Selection Lemma (see, e.g., \cite[1.17]{vDDM}).
\item[(ii)] Real analytic sets are arc-symmetric.
\item[(iii)] Let $\Omega$ be a bounded, connected, subanalytic, real analytic manifold in $\R^n$, let $\G$ be a connected, smooth, subanalytic subset of $\Omega$, and let $E\subset\Omega$ be subanalytic and arc-symmetric. Then, by \cite[Lem.\,2.3]{A2},
\[
\G\not\subset E\Longrightarrow \dim_\R(\G\cap E)<\dim_\R\G\,.
\]
\end{itemize}
\end{remark}

\begin{lemma}
\label{lem:image-arcsym}
Let $M$ be a real analytic manifold, let $f:M\to\R^n$ be an arc-analytic subanalytic function, and let $E\subset M$ be an arc-symmetric subanalytic set. Let $\G_f\subset M\times\R^n$ denote the graph of $f$, and let $\bar{f}:M\ni x\mapsto(x,f(x))\in\G_f$. Then, $\bar{f}(E)$ is an arc-symmetric subanalytic subset of $M\times\R^n$, and $\dim_\R{\bar{f}(E)}=\dim_\R{E}$.
\end{lemma}

\begin{proof}
The set $\bar{f}(E)$ is subanalytic in $M\times\R^n$, as the intersection of subanalytic sets $\G_f\cap(E\times\R^n)$.
For the proof of arc-symmetry, let $\gamma=(\gamma_1,\gamma_2):(-1,1)\to M\times\R^n$ be an analytic arc, with $\gamma_1:(-1,1)\to M$ and $\gamma_2:(-1,1)\to\R^n$, and such that $\gamma((-1,0))\subset \bar{f}(E)$. Then, $\gamma_1((-1,0))\subset E$, and hence $\gamma_1((-1,1))\subset E$, by arc-symmetry of $E$. By arc-analyticity of $f$, $f\circ\gamma_1$ is real analytic. The real analytic functions $f\circ\gamma_1$ and $\gamma_2$ coincide on $(-1,0)$, and hence on the whole interval $(-1,1)$. It follows that $\gamma((-1,1))\subset \bar{f}(E)$, as required.
The equality of dimensions of $\bar{f}(E)$ and $E$ follows from the fact that $\bar{f}$ is a subanalytic injection.
\end{proof}

\begin{corollary}
\label{cor:gen-analytic}
Let $M$ be a real analytic manifold and let $f:M\to\R^n$ be an arc-analytic subanalytic function. Then, there exists an arc-symmetric subanalytic set $\Sg_f\subset\G_f$ such that $\dim_\R\Sg_f\leq\dim_\R\G_f-2$ and $\G_f\setminus\Sg_f$ is a real analytic subset of the manifold $(M\times\R^n)\setminus\Sg_f$.
\end{corollary}

\begin{proof}
By \cite[Thm.\,3.1]{KP}, the non-analyticity locus $S(f)$ of $f$ is an arc-symmetric subanalytic set in $M$, of dimension $\dim_\R{S(f)}\leq\dim_\R{M}-2$. By Lemma~\ref{lem:image-arcsym}, $\Sg_f\coloneqq \bar{f}(S(f))$ is arc-symmetric, subanalytic, and $\dim_\R\Sg_f\leq\dim_\R\G_f-2$. Now, for any $(p,f(p))\in\G_f\setminus\Sg_f$, there is an open neighbourhood $U$ of $p$ in $M$ such that $f=(f_1,\dots,f_n)$ is analytic on $U$. Then, the set
\[
\G_f\cap(U\times\R^n)=\{(x,y_1,\dots,y_n)\in M\times\R^n:y_1-f_1(x)=\dots=y_n-f_n(x)=0\}
\]
is a real analytic subset of $U\times\R^n$.
\end{proof}
\smallskip

\subsection*{CR manifolds and CR functions}

Given an $\R$-linear subspace $L$ in $\C^n$ of dimension $d$, one defines the {\it CR dimension} of $L$ to be the largest $k$ 
such that $L$ contains a $\C$-linear subspace of (complex) dimension $k$. Clearly, $0\le k\le \left[\frac{d}{2}\right]$.
A $d$-dimensional real-analytic submanifold $M$ of an open set in $\C^n$ is called an {\it embedded \it CR manifold} of CR dimension $k$, if the tangent space $T_p M$ has CR dimension $k$ for every point $p\in M$ (the $k$-dimensional complex vector subspace of $T_p M$ will be then denoted by $H_pM$). We write 
$\dim_{CR}M=k$. If $k=0$, $M$ is called a \emph{totally real} submanifold.
The integer $l:=d-2k$ is called the \emph{CR codimension} of $M$, and the pair $(k,l)$ is the \emph{type} of $M$. A CR submanifold $M$ in $\C^n$ is called \emph{generic} when $n=k+l$, where $M$ is of type $(k,l)$. 

It is possible to define \emph{abstract} CR structure intrinsically on a smooth or real-analytic manifold; see, e.g.,~\cite{Bo}. While in the smooth case an abstract CR manifold does not in general admit a (local) CR embedding into a complex manifold, this is always possible when $M$ is a real-analytic CR manifold. The local embedding was proved by Andreotti-Hill~\cite{AH}, while the global embedding was obtained by Andreotti-Fredricks~\cite{AF}.

The notion of \emph{CR function} is usually defined in terms of tangential Cauchy-Riemann equations as follows. Given a real-analytic CR submanifold $M$ in an open set in $\C^n$, a smooth complex vector field $X$ on $M$ is called an anti-holomorphic CR vector field if $X_p \in H^{0,1}_pM$ for every $p \in M$, where $H^{0,1}_pM = (H_pM \otimes \mathbb{C}) \cap T^{0,1}\mathbb{C}^n$. A $C^1$-smooth function $f : M \to \mathbb{C}$ is CR if $Xf \equiv 0$ for every such vector field $X$ on $M$. In local coordinates of $\mathbb{C}^n$, such a vector field takes the form $X = \sum_{j=1}^n c_j \frac{\partial}{\partial \bar{z}_j}$, where the coefficients $c_j$ are smooth functions on $M$.

In order to use Theorem~\ref{thm:glob-CR-map} in its full generality, we shall need a more general definition that does not require smoothness of $f$. This can be done in terms of distributions. Suppose that $M$ is of type $(k,l)$. A locally integrable function $f:M\to\C$ is called a \emph{CR function} if
\[
\int_Mf\overline{\partial}\alpha=0
\]
for any differential form $\alpha$ of bidegree $(k,k+l-1)$ with compact support.
A \emph{CR mapping} $f=(f_1,\dots,f_N):M\to\C^N$ is one all of whose components $f_j$ are CR functions.

The following lemma will be used in the proof of Theorem~\ref{thm:glob-CR-map}.

\begin{lemma}[{\cite[Lem.\,5.1]{Sh}}]
\label{lem:Sh-lem-5-1}
Let $M$ be a smooth generic submanifold of an open set in $\mathbb C^n$, of positive CR dimension and codimension. Let $S\subset M$ be a smooth submanifold with $\dim S < \dim M$. Then any function $h$ continuous on $M$ and CR on $M\setminus S$ is CR on $M$.
\end{lemma}
\smallskip

\subsection*{Holomorphic closure dimension}

For a semianalytic set $E$ in a complex manifold $M$ and a point $p\in\overline{E}$, the \emph{holomorphic closure dimension} of $E$ at $p$ (denoted $\dim_{\mathrm{HC}}E_p$) is the complex dimension of the smallest complex analytic germ at $p$ in $M$ that contains the germ $E_p$.

For the reader's convenience, we recall here the main results of \cite{AS} concerning the holomorphic closure dimension of real analytic and semianalytic sets used in this note.

\begin{theorem}[{cf.\,\cite[Thm.\,1.1]{AS}}]
\label{thm:AS-thm-1-1}
Let $R$ be an irreducible real analytic set in a complex manifold $M$, of dimension $d>0$. Then, there exists a real analytic subset $S\subset R$ of dimension less than $d$ such that the holomorphic closure dimension of $R$ is constant on $R\setminus S$.
\end{theorem}

\begin{corollary}[{cf.\,\cite[Cor.\,1.2]{AS}}]
\label{cor:AS-cor-1-2}
Let $R$ be a semianalytic set of dimension $d>0$ in a complex manifold $M$. Suppose one of the following conditions holds:
\begin{itemize}
\item[(a)] $R$ is contained in an irreducible real analytic set of the same dimension;
\item[(b)] $R$ is connected, pure-dimensional and is locally contained in a locally irreducible real analytic set of the same dimension.
\end{itemize}
Then, there exists a closed semianalytic subset $S\subset R$, of dimension less than $d$ and such that the holomorphic closure dimension of $R$ is constant on $R\setminus S$.
\end{corollary}


\section{Proof of Theorem~\ref{thm:glob-holo-map}}
\label{sec:pf-thm-holo}

We begin by proving the following result.

\begin{proposition}
\label{prop:glob-cplx-analytic}
Let $\Omega$ be a complex manifold, and let $E$ be a connected arc-symmetric semianalytic subset of $\Omega$ of pure dimension $d>0$. Suppose that every $\xi\in E$ admits an open neighbourhood $U^\xi$ in $\Omega$ and a C-irreducible C-analytic subset $R^\xi$ of $U^\xi$ of dimension $\dim_\R{R^\xi}=d$, such that $E\cap U^\xi\subset R^\xi$. Suppose further that $E_{\xi_0}$ is a complex analytic germ, for some $\xi_0\in E$. Then, $E$ is a complex analytic subset of $\Omega$.
\end{proposition}

\begin{proof}
By Corollary~\ref{cor:AS-cor-1-2}, there exists a closed semianalytic set $F\subset E$, of dimension $\dim_\R{F}<d$, such that $E\setminus F$ is of constant holomorphic closure dimension, say, equal to $h$. Since $E$ is of pure dimension $d$, then $F$ is nowhere dense in $E$. In particular, $E\setminus F$ contains points arbitrarily close to $\xi_0$. By assumption, there is an open neighbourhood $V$ of $\xi_0$ in $\Omega$ such that $E\cap V$ is complex analytic in $V$. It follows that $\dim_\C(E\cap V)=h$, and so $h=d/2$.

For every $\xi\in E$, let $U^\xi$ be an open neighbourhood of $\xi$ in $\Omega$ and let $R^\xi$ be an irreducible C-analytic subset of $U^\xi$, of dimension $\dim_\R{R^\xi}=d$, such that $E\cap U^\xi\subset R^\xi$. By Theorem~\ref{thm:AS-thm-1-1}, for every $\xi\in E$, there is a real analytic set $S^\xi$ in $U^\xi$, of dimension less than $d$, such that $R^\xi\setminus S^\xi$ is of constant holomorphic closure dimension. This constant holomorphic closure dimension must thus be $h$, for all $\xi$, since $E\cap U^\xi$ contains a nonempty open subset of $R^\xi\setminus S^\xi$.

Let now $\xi\in E$ be such that $E_\xi$ is a complex analytic germ. By the Claim in the proof of \cite[Thm.\,4]{DF}, there exists a complex analytic subset $Z^\xi$ of $U^\xi$ such that $(Z^\xi)_\xi=E_\xi$ and $Z^\xi\subset R^\xi$. As a complex analytic set, $Z^\xi$ is coherent, and hence C-analytic in $U^\xi$. Moreover,
\[
\dim_\R{Z^\xi}=2\dim_\C{Z^\xi}\geq2\dim_\C{E_\xi}=2h=d=\dim_\R{R^\xi}\,.
\]
Thus, by C-irreducibility of $R^\xi$ and Remark~\ref{rem:C-analytic}(iv), we have $Z^\xi=R^\xi$. Now, the arc-symmetric subset $E\cap U^\xi$ of $U^\xi$ contains a non-empty open subset of the connected (by irreducibility of $Z^\xi$) manifold $\Reg_\C{Z^\xi}$, and hence $\Reg_\C{Z^\xi}\subset E\cap U^\xi$, by Remark~\ref{rem:AR-facts}(iii). It follows that $E\cap U^\xi\supset Z^\xi$, as $E$ is closed. Since we also have $E\cap U^\xi\subset R^\xi=Z^\xi$, then $E\cap U^\xi=Z^\xi$ is complex analytic.

Now, by connectedness of $E$, the complex analytic set $Z^{\xi_0}\subset U^{\xi_0}$ can be extended to a complex analytic subset $Z$ of $\bigcup_{\xi\in E}U^\xi$, such that $E=E\cap\bigcup_{\xi\in E}U^\xi=Z$. Since $E$ is closed in $\Omega$, $Z$ is a complex analytic subset of $\Omega$.
\end{proof}

\subsubsection*{Proof of Theorem~\ref{thm:glob-holo-map}}
Let $\G_f\subset M\times N$ denote the graph of $f$. Let $\pi_M:M\times N\to M$ and $\pi_N:M\times N\to N$ be the canonical projections.
By Corollary~\ref{cor:gen-analytic}, $\G_f$ contains a closed subanalytic set $\Sg_f$, such that $\G_f\setminus\Sg_f$ is real analytic and $\dim_\R\Sg_f\leq\dim_\R\G_f-2$. The set $\Sg_f$ is nowhere dense in $\G_f$ and $\G_f\setminus\Sg_f$ is connected, as a homeomorphic image of the connected manifold $M\setminus S(f)$.

We claim that $\G_f$ is a complex analytic subset of $M\times N$. By a theorem of Peterzil-Starchenko \cite[Thm.\,4.2]{PS}, it suffices to show that for every open $\Omega\subset M\times N$,
\[
\dim_\R\Sing_\C(\Omega\cap\G_f)\leq\dim_\R(\Omega\cap\G_f)-2\,,
\]
where $\Sing_\C(X)$ denotes the set of points $\xi\in X$ at which $X_\xi$ is not a germ of a complex manifold. (By convention, $\dim\varnothing=-\infty$.)
Note that $\Sing_\C\G_f$ is subanalytic in $M\times N$ (see, e.g., \cite[Fact\,2.6(ii)]{PS}).
It is thus enough to show that $\G_f\setminus\Sigma_f$ is a complex analytic subset of the complex manifold $(M\times N)\setminus\Sigma_f$, of pure dimension $\dim_\R\G_f/2=\dim_\C{M}$.

The latter follows immediately from Proposition~\ref{prop:glob-cplx-analytic}, because every point $\xi\in\G_f\setminus\Sg_f$ admits an open neighbourhood $\Omega^\xi$ in $(M\times N)\setminus\Sg_f$ such that $(\G_f\setminus\Sg_f)\cap\Omega^\xi$ is a C-irreducible C-analytic subset of $\Omega^\xi$.
Indeed, for any such $\xi=(p,f(p))$, one can choose connected open neighbourhoods $U^p$ of $p$ in $M$ and $V^p$ of $f(p)$ in $N$ so that $\Omega^\xi\coloneqq U^p\times V^p$ is contained in $(M\times N)\setminus\Sg_f$ and $(\G_f\setminus\Sg_f)\cap\Omega^\xi$ is C-analytic in $\Omega^\xi$. By Remark~\ref{rem:C-analytic}, every C-analytic set admits a locally finite decomposition into C-irreducible C-analytic sets. Therefore, $(\G_f\setminus\Sg_f)\cap\Omega^\xi$ is C-irreducible as an analytic image of a connected manifold $\pi_M((\G_f\setminus\Sg_f)\cap\Omega^\xi)=U^p$ by $(\pi_M|_{\G_f})^{-1}$. 

Finally, notice that $\G_f\setminus\Sg_f$ contains a complex analytic germ of dimension $\dim_\R\G_f/2$. Indeed, if $p_0\in M$ is such that $f$ is holomorphic in an open neighbourhood $U$ of $p_0$ in $M$, then $U\subset M\setminus S(f)$ and the set
\[
\G_f\cap(U\times N)=\{(p,q)\in M\times N:q=f(p)\}
\]
is a complex analytic subset of $U\times N$.

Now, since $\Gamma_f$ is a complex analytic subset of $M\times N$, then $\pi_M\!|_{\Gamma_f}:\Gamma_f\to M$ is a bijective holomorphic mapping, and hence a biholomorphism (see, e.g., \cite[\S\,3.3, Prop.\,3]{Chi}). It follows that $f=\pi_N\!|_{\Gamma_f}\circ(\pi_M\!|_{\Gamma_f})^{-1}$ is holomorphic on the whole $M$.
\qed


\section{Proof of Theorem~\ref{thm:glob-CR-map}}
\label{sec:pf-thm-CR}

Assume that $M$ is of real dimension $d$ and CR-dimension $m>0$. If $M$ is a complex manifold, then CR maps on $M$ are precisely the holomorphic maps, and so
the result follows from Theorem~\ref{thm:glob-holo-map}. Assume then that $M$ has positive CR-codimension; i.e., $d - 2m >0$.
Since $M$ is real analytic, by Andreotti-Fredrics~\cite{AF}, there exists a complex manifold $X$, of 
complex dimension $ n= d-m$, such that $M$ admits a real analytic generic CR embedding into $X$. 
Thus, we may assume that 
$M \subset X$ and $\Gamma_f \subset X \times \C^N$.
Let $\pi:\G_f\to M$ be the canonical projection.

Suppose that $U$ is an open subset of $M$ on which $f$ is CR. Then, there exists a nonempty
open set $\tilde U \subset U$ such that $f: \tilde U\to\C^N$ is a real analytic CR map.
By \cite{T}, every component $f_j$ of $f|_{\tilde U}$
extends to a function $F_j$ holomorphic in some neighbourhood $\Omega$ of
$\tilde U$ in $X$. The graph of $F=(F_1,\dots,F_N)$ is a complex
analytic set in $\Omega\times\C^N$ of dimension~$n$ containing $\G_f$. 
One can readily verify that $\G_F$ is the smallest complex analytic set which
contains $\G_f$. Hence,
\begin{equation}
\label{eq:HC-dim-is-n}
\dim_{\mathrm{HC}}{\G_f}_{(p,f(p))}=n,\quad\mathrm{for\ all\ }p\in\tilde{U}\,.
\end{equation}
\smallskip

By Corollary~\ref{cor:gen-analytic}, there is a closed subanalytic set $\Sg_f$ in the graph $\G_f$ of $f$, such that $\G_f\setminus\Sg_f$ is a real analytic subset of 
$(M\times\C^N)\setminus\Sg_f$, of pure dimension $d$, and $\dim\Sg_f\leq d-2$.
As in the proof of Theorem~\ref{thm:glob-holo-map}, $M\setminus\pi(\Sg_f)=M\setminus S(f)$ is a connected real analytic manifold, and hence $\G_f\setminus\Sg_f$ is connected.

We first claim that $f$ is CR on $M\setminus S(f)$.  
The proof below relies on Theorem~\ref{thm:AS-thm-1-1} and is a simplification of the argument given 
in~\cite[Thm.\,1.3]{Sh}. 

As in the proof of Theorem~\ref{thm:glob-holo-map}, $\G_f\setminus\Sg_f$ is locally irreducible. 
Since $\G_f\setminus\Sg_f$ is also connected, it is an irreducible real analytic set in $M\times\C^N$. It follows from Theorem~\ref{thm:AS-thm-1-1} that 
there exists a closed real analytic subset $\Sg\subset \G_f\setminus\Sg_f$, of dimension
$\dim_\R\Sg< d$, such that the holomorphic closure dimension of $\G_f \setminus (\Sg_f\cup\Sg)$ is constant, and hence equal to $n=d-m$ (by \eqref{eq:HC-dim-is-n}). 
Let 
$$
(p,f(p)) \in (M \times \mathbb C^N) \cap (\G_f \setminus (\Sg_f\cup\Sg) ),
$$
and let $Y\subset X\times \mathbb C^N$ be a representative 
of the  unique smallest complex analytic germ of dimension $n$ that contains the germ of $\G_f$ at $(p,f(p))$. 
Let $\pi_X : Y \to X$ and $\pi_N : Y \to \C^N$ be the canonical projections.
Since $M$ is generic, $\pi_X $ is surjective and generically finite. (Indeed, for otherwise all fibres of $\pi_X$ would be of positive dimension, and hence $M\subset\pi_X(Y)$ would be contained in the union of countably many complex submanifolds of dimensions less than $n$.)
It follows that
there exists a complex analytic subset $V$ of a neighbourhood $U^p\subset X$ of $p$, $\dim_\C V \le n-1$, such that 
for any point $q \in U^p \setminus V$, there exits a neighbourhood $U^q \subset U^p \setminus V$ such that the inverse projection
$$
\pi^{-1}_X|_{U^q} : U^q \to Y
$$
splits into a finite collection of  biholomorphic maps $g_j: U^q \to Y$, $j=1,\dots,\nu$. 
In other words, $V$ contains all points in $U^p$ that have fibres
of positive dimension under the projection $\pi_X$ or belong to the branch locus of $\pi_X$.
Note that since $M$ is generic in $X$, the variety $V$ cannot contain $M$ near $p$, 
and therefore $\dim_\R (M \cap V) \le d-1$. 
The graph of one of the maps $g_j$ must contain the graph of $f$ over $U^q$, and so it  
is a holomorphic extension of $f|_{U^q}$. This shows that $f$ is a CR map near $q$, hence $f$ is CR on 
$(M \cap U^p) \setminus V$.

Let $\mathcal S\subset M \setminus \pi(\Sg_f)$ be the union of $\pi(\Sg)$ and all $V\cap M$, where $V$ are locally defined analytic sets as 
discussed above. To prove the claim it suffices to show that $f$ is CR at the points of $(M\setminus \pi(\Sg_f)) \cap \mathcal S$. 
Since $\mathcal S$ is subanalytic, there exists a stratification of $\mathcal S$ into real analytic manifolds 
$S_1, S_2,\dots$ of dimension at most $d-1$. The result now follows from Lemma~\ref{lem:Sh-lem-5-1}.
Indeed, we may apply the lemma inductively to every stratum of the stratification of $\mathcal S$ starting with the strata of the top dimension. 

Finally, to show that  $f$ is CR everywhere on $M$ observe that $\pi(\Sg_f)=S(f)$ is a subanalytic subset of $M$ of dimension at most $d-2$. This implies that it admits a similar stratification into smooth manifolds on $M$, and so the lemma can be used inductively again to show that $f$ is CR at the points of $\pi(\Sg_f)$.
\qed

\begin{remark}
\label{rem:alternative-assumptions}
Note that the above proof may be easily adapted to show that the conclusion of Theorem~\ref{thm:glob-CR-map} holds under an alternative set of assumptions. Namely, one may assume that $f:M\to\C^N$ is continuous, has a CR germ at some point, and the graph $\G_f$ is semianalytic and (locally) contained in a (locally) irreducible real analytic set of dimension $d=\dim{M}$.
\end{remark}


\section{Examples of arc-analytic subanalytic functions}
\label{sec:examples}

The following example is due to Kurdyka \cite{Ku-priv}.

\begin{example}
\label{ex:Kurdyka}
Consider a function $f:\R^2\to\R$ defined as
\[
f(x_1,x_2)=
\begin{cases}
x_1\cdot\exp\left(\frac{x_1^2}{x_1^2+x_2^2}\right), &\mathrm{if\ } x_1^2+x_2^2\neq0\\
0\,, &\mathrm{otherwise}.
\end{cases}
\]
The composition of $f$ with the blowing-up of the origin in $\R^2$ is real analytic, and hence $f$ is arc-analytic. To see that the graph $\Gamma_f$ of $f$ is a subanalytic subset of $\R^3$, consider the blowing-up of the origin in $\R^3$, $\sigma:M\to\R^3$. It is easy to see that the strict transform $S$ of $\Gamma_f$ by $\sigma$ is a real analytic subset of $M$. Thus, $\Gamma_f=\sigma(S)$ is the image of a real analytic set by a proper real analytic mapping.

Further, observe that $\Gamma_f$ is not semianalytic. Indeed, let $h$ be an arbitrary real analytic function in a neighbourhood $U$ of the origin in $\R^3$, and suppose that $h$ vanishes on $\Gamma_f\cap U$. Write $h=\sum_{\nu=1}^\infty{h_\nu}$, where each nonzero $h_\nu$ is a homogenous polynomial of degree $\nu$. Notice that $\Gamma_f$ is a real cone, that is, if $\xi\in\Gamma_f$ and $t\in\R$ then $t\xi\in\Gamma_f$. It follows that $h_\nu$ vanishes on $\Gamma_f\cap U$, for all $\nu$. In particular, for every nonzero $h_\nu$, $h_\nu(x_1,x_2,1)$ is a nonzero polynomial, which vanishes on a transcendental curve
\[
\Gamma_f\cap\{x_3=1\}\ =\ \{1=x_1\cdot\exp\left(\frac{x_1^2}{x_1^2+x_2^2}\right)\}\,;
\]
a contradiction. This proves that $h$ is identically zero. Consequently, the two-dimensional set $\Gamma_f$ is not contained around $0\in\R^3$ in any two-dimensional real analytic set, which means that $\Gamma_f$ is not semianalytic at the origin (see, e.g., \cite[Thm.\,2.13]{BM1}).
\end{example}

\begin{remark}
\label{rem:no-arcan-complexification}
In the context of Theorem~\ref{thm:glob-holo-map}, it is worth noting that the function $f$ from the above example does not admit an arc-analytic extension to the complexification of $\R^2$. For, let $M$ be an arbitrarily small connected open neighbourhood of $\R^2$ in $\C^2$, and let the function $\vp:M\to\C$ be the restriction to $M$ of the function
\[
\vp(z_1,z_2)=
\begin{cases}
z_1\cdot\exp\left(\frac{z_1^2}{z_1^2+z_2^2}\right), &\mathrm{if\ } z_1^2+z_2^2\neq0\\
0\,, &\mathrm{otherwise}.
\end{cases}
\]
Let $\e>0$ be small enough so that the point $(z_1,z_2)=(\e+i\e,\e-i\e)$ is contained in $M$. Note that this point lies in the locus $\{z_1^2+z_2^2=0\}$. Let $\gamma:(-\delta,\delta)\to M\subset\R^4$ be an analytic arc defined as $\gamma(t)=(t+\e,t+\e,t+\e,t-\e)$. We claim that then $\vp\circ\gamma$ is not continuous at $t=0$.
Indeed, writing $\vp=\vp_1+i\vp_2$, we have
\[
\vp_1(\gamma(t))=(t+\e)\cdot\exp(u(t))\cdot(\cos(v(t))+\sin(v(t)))\,,
\]
where
\[
u(t)=\frac{\e(t+\e)^2}{2t(\e^2+(t+\e)^2)}\mathrm{\ \ \ and\ \ \ }
v(t)=\frac{(t+\e)^3}{2t(\e^2+(t+\e)^2)}\,.
\]
It is easy to see that $\displaystyle{\lim_{t\to0}\vp_1(\gamma(t))}$ does not exist.
It follows that $\vp\circ\gamma$ is not analytic, and hence $\vp$ is not arc-analytic.
On the other hand, the graph of $\vp$ is subanalytic, by the same argument as in Example~\ref{ex:Kurdyka}.
\end{remark}
\smallskip

In contrast with Theorem~\ref{thm:glob-holo-map}, an arc-analytic mapping on a CR manifold $M$ need not be real analytic on the whole $M$, as the following example shows.

\begin{example}
\label{ex:2}
Let $M=\R^2\times\C$, let $f:\R^2\to\R$ be the function from Example~\ref{ex:Kurdyka}, and let $g:M\to\C$ be defined as $g(x_1,x_2,z_3)=f(x_1,x_2)$.
Then, $M$ is a generic CR submanifold in $\C^3$ (of dimension $4$ and CR-dimension $1$), and $g$ is an arc-analytic CR mapping on $M$. Indeed, for $g$ is holomorphic along every complex line contained in $M$, as the only such lines are of the form $\{x_1\equiv c_1,x_2\equiv c_2\}$ and $g$ is constant along these lines.

On the other hand, the graph of $g$ is not real analytic at the origin in $M\times\C$, by Example~\ref{ex:Kurdyka}. It follows that $g$ is not analytic at $(0,0,0)$.
\end{example}

\begin{example}
\label{ex:3}
Interestingly, even assuming that the graph $\G_g$ be real analytic in Theorem~\ref{thm:glob-CR-map} does not imply analyticity of the mapping $g$.
Consider, for instance, $M=\R^2\times\C$ and $g(x_1,x_2,z_3)=\sqrt[3]{x_1^6+x_2^6}$. Here, $\G_g$ is the real algebraic set $\{(x_1,x_2,z_3,w)\in M\times\C:w^3=x_1^6+x_2^6\}$, and $g$ is CR on $M$ for the same reason as in Example~\ref{ex:2}. Moreover, $g$ is arc-analytic, as it becomes analytic after composition with the blowing-up of the origin in $\R^2$. Nonetheless, $g$ is not real analytic at $(0,0,0)\in M$.
\end{example}

\bibliographystyle{amsplain}

\end{document}